\documentclass{amsart}
\usepackage{amssymb, mathrsfs, bbm}
\usepackage[T1]{fontenc}
\usepackage[sc, osf]{mathpazo}
\selectfont
\usepackage[protrusion=true, expansion=true]{microtype}

\usepackage[all, arc]{xy}
\SelectTips{eu}{}
\entrymodifiers={+!!<0pt,\fontdimen22\textfont2>}

\usepackage[pdftex, colorlinks=true, linkcolor=blue, citecolor=blue, linktocpage]{hyperref}

\makeatletter
\def\tagform@#1{\maketag@@@{\bfseries(\ignorespaces#1\unskip\@@italiccorr)}}
\renewcommand{\eqref}[1]{\textup{{\normalfont(\ref{#1}}\normalfont)}}
\makeatother

\swapnumbers
\theoremstyle{plain}

\newtheorem{lemma}[equation]{Lemma}
\newtheorem{prop}[equation]{Proposition}

\theoremstyle{definition}

\theoremstyle{definition}

\def\cF{\mathcal{F}}
\def\cG{\mathcal{G}}

\def\cL{\mathcal{L}}

\def\11{\mathbf{1}}

\def\CC{\mathbf{C}}

\newcommand{\mapright}[1]{\xrightarrow{#1}}

\title{A Remark on Ehresmann's Fibration Theorem}
\author{R. Virk}
\begin{document}
\maketitle
\renewcommand{\thesubsection}{\arabic{subsection}}
If $f\colon Z \to Y$ is a smooth proper morphism of smooth varieties, and $\cL$ a local system on $Z$, then the sheaves $R^qf_*\cL$ are local systems on $Y$. This is typically seen as a consequence of Ehresmann's Theorem - $f$ is a topological fiber bundle over each component of $Y$ (\cite[Theorem 9.3]{Vo} is a convenient reference).
This note records that the cohomological consequence holds without the smooth assumption on $Y$ or $Z$.

\subsection*{Conventions}
A `sheaf' means a `sheaf of vector spaces over some fixed field', and `variety' = `separated reduced scheme of finite type over $\mathrm{Spec}(\CC)$'. Sheaves on varieties are with respect to the complex analytic site. A proper map of topological spaces is a separated and universally closed map.

J-L.\ Verdier asserts the following without the locally connected hypothesis \cite[Lemme 2.2.2]{V}. I was unable to understand his proof without this assumption.
\begin{lemma}Let $p\colon X \to Y$ be a proper surjective map of topological spaces. Assume $X$ is locally connected. Let $\cF$ be a sheaf on $Y$ with finite dimensional stalks. If $p^*\cF$ is a local system, then so is $\cF$.
\end{lemma}

\begin{proof}
	Let $y\in Y$. 
	The stalk $\cF_y$ is finite dimensional, so there exist sections $s_1, \ldots, s_n$, of $\cF$ over some open neighborhood of $y$, which restrict to a basis of $\cF_y$. Since our problem is local, we may assume this neighborhood is all of $Y$. Let $\cG$ be the constant sheaf on $Y$ with stalk $\mathrm{span}\{s_1,\ldots,s_n\}$. Then the evident map
$u\colon \cG \to \cF$ induces an isomorphism $\cG_y \mapright{\sim}\cF_y$.
	Consequently, $p^*u$ induces isomorphisms:
	\[ (p^*\cG)_x \mapright{\sim} (p^*\cF)_x \quad \mbox{for all $x\in p^{-1}(y)$}.\]
	For a locally connected space, the set of points at which a morphism of local systems induces an isomorphism on stalks defines an open set.
	Hence, the set $V\subset X$ of points at which $p^*u$ induces isomorphisms on stalks is open.
	As $p$ is proper, $U = Y - f(X-V)$ is an open neighborhood of $y$. As $p$ is surjective, $u$ yields an isomorphism $\cG|_U \mapright{\sim} \cF|_U$.
\end{proof}

\begin{prop}
	Let $f\colon Z \to Y$ be a smooth and proper morphism of varieties. Let $\cL$ be a local system on $Z$ with finite dimensional stalks. Then the sheaves $R^qf_*\cL$ are local systems.
\end{prop}

\begin{proof}
	Resolution of singularities (the version in \cite{BP} suffices), the Lemma and proper base change reduce us to the situation where $Z$ and $Y$ are smooth. Here the usual form of Ehresmann's Theorem applies.
\end{proof}	


\begin{thebibliography}{99}
	\bibitem[BP]{BP} {\sc F. Bogomolov, T. Pantev}, {\em Weak Hironaka Theorem}, arXiv:alg-geom/9603019v2.
	\bibitem[Ve]{V} {\sc J-L. Verdier}, {\em Classe d'Homologie associ\'ee un Cycle}, Asterisque \textbf{36-37}, p.\ 101-151 (1976).
	\bibitem[Vo]{Vo} {\sc C. Voisin}, {\sl Hodge Theory and Complex Algebraic Geometry I}, Cambridge Studies in Math. \textbf{76} (2002).
\end{thebibliography}
\end{document}